\documentclass[12pt]{article}
\usepackage{amssymb,amsmath,amsthm}
\usepackage{graphicx}
\usepackage{subcaption}
\usepackage{fullpage}
\usepackage{scrextend}
\usepackage{siunitx}
\usepackage{enumerate}
\usepackage{tikz,calc}
\usepackage{tikz-cd}
\usetikzlibrary{automata,positioning}
\usepackage{subcaption}
\usetikzlibrary{calc}
\usepackage{epsfig,graphicx}
\usepackage{listings}
\usepackage{enumerate}
\usepackage{float}
\usepackage{diagbox}
\usepackage{url}
\usepackage{makecell}
\usepackage{float}

 \newtheorem{theorem}{Theorem}
\newtheorem{prop}[theorem]{Proposition}
\newtheorem{lemma}[theorem]{Lemma}

\newtheorem{conjecture}[theorem]{Conjecture}

\newcommand{\ol}{\overline}

\newcommand{\ind}{\mathbf{1}}

\newcommand{\fl}{\text{fl}}

\numberwithin{equation}{section}

\title{Erd\H{o}s' minimum overlap problem}
\author{Ethan Patrick White}
\date{}

\begin{document}

\maketitle

\begin{abstract}
We obtain a substantially improved lower bound for the minimum overlap problem asked by Erd\H{o}s. Our approach uses elementary Fourier analysis to translate the problem to a convex optimization program. 

\end{abstract}

\let\thefootnote\relax\footnote{The author is grateful for support from the Killam Trusts, NSERC, and UBC.}
\let\thefootnote\relax


\section{Introduction}

In 1955 Erd\H{o}s posed the following problem~\cite{E}. Let $n$ be a positive integer and $A,B \subset [2n]$ be a partition of $[2n]$ such that $|A|=|B| = n$. For any such partition and integer $-2n<k<2n$, define $M_k$ to be the number of solutions $(a,b) \in A  \times B$ to $a-b = k$. Estimate the size of the function
\[ M(n) = \min_{A\cup B = [2n]} \max_{-2n<k<2n} M_k,\]
where the minimum is taken over all partitions of $[2n]$ into equal-sized sets. Erd\H{o}s proved that $M(n)>n/4$, a result that can be obtained by the following averaging argument. The sum over all $-2n<k<2n$ of $|M_k|$ is exactly $|A \times B| = n^2$, and so the average value of $|M_k|$ exceeds $n/4$. On the other hand, we can take $A$ to be $[n/2,3n/2]$ giving the upper bound $M(n) \leq n/2$. The minimum overlap problem appears in Richard Guy's renown book, \emph{Unsolved Problems in Number Theory}. See his book for a brief survey of the progress made by many authors on improved estimates of $M(n)$~\cite{RKG}.

Haugland proved that the limit 
\[ \mu:= \lim_{n \to \infty} \frac{M(n)}{n},\]
exists~\cite{H1}. We will refer to $\mu$ as the \emph{minimum overlap constant}. Prior to this work, the best estimates of $\mu$ were 
\[ 0.35639395869 \approx \sqrt{4-\sqrt{15}}  \leq \mu \leq 0.3809268534330870 .\]
The lower bound above is due to Moser~\cite{M}, and the upper bound is due to Haugland~\cite{H2}. Moser and Murdeshwar were the first to study the following function analogue of Erd\H{o}s' original problem. For all measurable functions $f \colon [-1,1] \to [0,1]$ define the complementary function $g \colon [-1,1] \to [0,1]$ such that $f(x) + g(x) = 1$ for all $x \in [-1,1]$. Estimate the value of
\begin{equation}\label{contver} \inf_{f} \sup_{x \in [-2,2]} \int_{-1}^1 f(t)g(x+t) \ dt,\end{equation}
where the minimum is taken over all measurable $f \colon [-1,1] \to [0,1]$ satisfying $\int_{-1}^1 f(x) \ dx = 1$. A key step in Haugland's method is a theorem of Swinnerton-Dyer proving that \eqref{contver} is in fact also $\mu$, see~\cite{H1} for the proof. It will be easiest for us to work with \eqref{contver} as our definition of $\mu$. In this work we obtain a significant improvement on the lower bound of $\mu$, using elementary Fourier analysis combined with convex programming.

\begin{theorem}\label{T1} The minimum overlap constant $\mu$ is lower bounded by $0.379005$.
\end{theorem} 

The bound in Theorem~\ref{T1} can certainly be improved with more computation time for our convex program. The upper and lower bounds for $\mu$ now differ by $0.5\%$. 




\section{Outline}

Throughout this work, $f(x),g(x),M(x)$ will always denote measurable functions satisfying the relationships
\[ f \colon [-1,1] \to [0,1] \text{ such that } \int_{-1}^1 f(x) \ dx = 1;\]
\[ g \colon [-1,1] \to [0,1] \text{ such that } f(x) + g(x) = 1 \text{ for all } x \in [-1,1];\]
\begin{equation}\label{fgm} M \colon [-2,2] \to [0,1] \text{ such that } M(x) = \int_{-1}^1 f(t) g(x+t) \ dt.\end{equation}

The minimum overlap problem is to determine the largest $\mu$ such that $\| M \|_\infty \geq\mu$ for all functions $M$ satisfying \eqref{fgm}. Lower bounds on $\mu$ can be obtained by observing properties held by $M(x)$. For example, a first simple property held by $M(x)$ is 
\begin{equation}\label{outmass} \int_{-2}^2 M(x) \ dx =   \int_{-2}^2  \int_{-1}^1 f(t) g(x+t) \ dtdx =  \int_{-1}^1 f(t)  \int_{-2}^2 g(x+t) \ dx dt =1.  \end{equation}
Therefore the average value of $M(x)$ is at least 0.25 and so $\mu \geq 0.25$. The discrete version of this argument was already mentioned in the introduction. A second property held by $M(x)$, and the key insight in Moser and Murdeshwar's method \cite{M}, \cite{M2} is that
\begin{equation}\label{outvar} \int_{-2}^2 (x-E(M))^2 M(x) \ dx \leq 2/3, \quad  \text{where} \quad E(M) = \int_{-2}^2 x M(x) \ dx, \end{equation}
is the expected value of $M(x)$. In other words, the variance of $M(x)$ is upper bounded by $2/3$. The variance of a function is minimized when as much mass as possible is centred at its mean. Therefore the variance of $M(x)$ is at least the variance of
\[ \tilde{M}(x) = \begin{cases} \mu & \text{if } -\frac{1}{2\mu} \leq x \leq \frac{1}{2\mu} \\ 0 & \text{otherwise.} \end{cases}\]
The variance of $\tilde{M}$ is $1/(12\mu^2)$. Since $1/(12\mu^2) \leq 2/3$ we have $\mu \geq 1/\sqrt{8}$. 

The key idea behind our improvement is that Fourier analysis can be used to construct an infinite number of new properties satisfied by $M(x)$. The most important new property we find and use is that all even cosine Fourier coefficients of $M(x)$ are nonpositive, i.e. 
\begin{equation}\label{outcos} \int_{-2}^2 \cos(\pi k x) M(x) \ dx \leq 0 , \quad \text{for all } k \geq 1.\end{equation}
The way we take advantage of these new properties is by constructing a linear program where the variables represent the average value of $M(x)$ over small intervals. With this programming technique we can transfer \eqref{outmass}, \eqref{outvar}, and \eqref{outcos} into constraints of a linear program. This short program is the subject of Section~\ref{simpsec}. Under the assumption that an optimal $f(x)$ is even, the output of this program proves $\mu \geq 0.375$. 

In order to prove our larger lower bound in Theorem~\ref{T1} and remove the assumption that an optimal $f(x)$ is even, we use a more complicated convex program. This is the subject of Section~\ref{fullconsec}. The new properties of $f(x),M(x)$ pairs we find are derived in Section~\ref{setup}. These new properties become the constraints used in the linear program of Section~\ref{simpsec} and the convex program of Section~\ref{fullconsec}. In Section~\ref{fullconsec} we discuss the data collected from the convex program and prove Theorem~\ref{T1}. 


\section{Set up}\label{setup}

In this section we derive properties about $f,g,M$ that will form key constraints in the programs used in later sections. 

\subsection{Properties from Fourier analysis}

The properties derived in this subsection relate to the Fourier coefficients of $f,g,M$. We first consider $f,g,M$ as functions on $[-2,2]$. In the case of $f$ and $g$ we define $f(x) = g(x) = 0$ for $x \not\in [-1,1]$. Their Fourier transforms are defined for all $k \in \mathbb{Z}$ and given by
\[ \hat{f}(k) = \frac{1}{4}\int_{-2}^2 e^{-\frac{\pi i}{2} kx} f(x) \ dx, \quad  \hat{g}(k) = \frac{1}{4}\int_{-2}^2 e^{-\frac{\pi i}{2} kx} f(x) \ dx, \quad  \hat{M}(k) = \frac{1}{4}\int_{-2}^2 e^{-\frac{\pi i}{2} kx} M(x) \ dx.  \]

\begin{lemma}\label{expfou} For all $f,M$ satisfying \eqref{fgm} and $k \in \mathbb{Z}\setminus\{0\}$ we have
\[ \hat{M}(k) = \frac{4}{k\pi } \sin \left(k\pi /2 \right) \ol{\hat{f}(k)} - 4|\hat{f}(k)|^2,\]
where $\hat{f},\hat{M}$ denote the Fourier transforms of $f,M$ on $[-2,2]$. 
\end{lemma}

\begin{proof} Denote the indicator function of $[-1,1]$ by $\ind_{[-1,1]}$. We have
\begin{align*}
\hat{M}(k) &= \frac{1}{4}\int_{-2}^2 e^{-\frac{\pi i}{2} kx} M(x) \ dx =\frac{1}{4}\int_{-2}^2 e^{-\frac{\pi i}{2} kx} \int_{-1}^1 f(t) g(x+t) \ dt \ dx \\
& = \frac{1}{4}\int_{-1}^1 e^{\frac{\pi i}{2} kt} f(t) \int_{-2}^2 e^{-\frac{\pi i}{2}k(x+t)} g(x+t) \ dx \ dt  = 4\ol{\hat{f}(k)}\hat{g}(k) \\
& = 4\ol{\hat{f}(k)} \left( \hat{\ind}_{[-1,1]}(k) - \hat{f}(k) \right) .
\end{align*}
Combining the above with 
\[ \hat{\ind}_{[-1,1]}(k) = \frac{1}{4} \int_{-1}^1 e^{-\frac{\pi i }{2} kx} =\frac{1}{k\pi } \sin \left(k\pi /2 \right),  \]
gives the claimed identity.
\end{proof}

It will be advantageous to work with real Fourier coefficients in our programs of the later sections. We will derive identities relating the coefficients of the sine-cosine Fourier series of $f(x)$ as a function on $[-1,1]$ with the sine-cosine Fourier series of $f(x)$ and $M(x)$ as functions on $[-2,2]$. Let 
\begin{equation}\label{trigf1} f(x) = \frac{1}{2} + \sum_{k =1}^\infty c_k \cos (k \pi x ) +\sum_{k =1}^\infty d_k \sin (k \pi x ) , \end{equation}
be the sine-cosine Fourier series of $f(x)$ on $[-1,1]$. It will be notationally helpful to put $c_0 = 1/2$. Also let
\[ f(x) = \frac{1}{4} + \sum_{k =1}^\infty a_k \cos (k \pi x/2 ) +\sum_{k =1}^\infty b_k \sin (k \pi x/2 ) , \quad \text{and}\]
\begin{equation}\label{trigM} M(x) = \frac{1}{4} + \sum_{k =1}^\infty A_k \cos (k \pi x/2 ) +\sum_{k =1}^\infty B_k \sin (k \pi x/2 ) \end{equation}
be the sine-cosine Fourier series of $f$ and $M$ on $[-2,2]$.

\begin{lemma}\label{fourel} Let $f,M$ be as in \eqref{fgm} and their sine-cosine Fourier series be as in \eqref{trigf1} and \eqref{trigM}. Then for all $m \geq 1$:

\begin{equation}\label{am} a_m = \begin{cases} \frac{1}{2}c_{m/2} & \text{if }m \text{ is even} \\ 
\frac{2m\sin(\pi m/2)}{\pi} \sum_{k=0}^\infty \frac{(-1)^k }{m^2-4k^2} c_k & \text{if }m \text{ is odd.} \end{cases}\end{equation}

\begin{equation}\label{bm} b_m = \begin{cases} \frac{1}{2}d_{m/2} & \text{if }m \text{ is even} \\ 
\frac{4\sin(\pi m/2)}{\pi} \sum_{k=1}^\infty \frac{k(-1)^k }{m^2-4k^2} d_k & \text{if }m \text{ is odd.} \end{cases}\end{equation}

\begin{equation}\label{Am} A_m = \frac{4 \sin \left(m\pi /2 \right)}{m \pi}  a_m - 2(a_m^2+b_m^2), \text{ and in particular } A_{2m} \leq 0.\end{equation}
\begin{equation}\label{Bm} B_m = -\frac{4}{m \pi}\sin\left( m\pi /2 \right) b_m, \text{ and in particular } B_{2m} = 0.\end{equation}

\end{lemma}

\begin{proof} Note that the denominators in equations~\eqref{am} and \eqref{bm} are never zero since $m$ is an odd integer while $4k^2$ is even. The convergence of the two infinite sums in \eqref{am} and \eqref{bm} is implied by Lemma~\ref{tailbounds} below. To derive \eqref{am} we use the following integral identities, valid for all $m \geq 1$.
\[ \int_{-1}^1 \cos(\pi m x/2) \cos (\pi k x) \ dx = \begin{cases} 1 & \text{ if } m \text{ is even and } k = m/2 \\ 0&  \text{ if } m \text{ is even and } k \neq m/2 \\ \frac{4m (-1)^k \sin(\pi m /2)}{\pi m^2 - 4\pi k^2 }& \text{ if } m \text{ is odd. } \end{cases}\]
We also have the identity $\int_{-1}^1 \cos(\pi m x/2) \sin (\pi k x) \ dx = 0$ for all $k,m \in \mathbb{Z}$. Combining these identities with 
\[ a_m = \frac{1}{2} \int_{-1}^1 \cos(\pi m x /2) f(x) \ dx, \]
gives \eqref{am}. Similarly, we have for all $m \geq 1$ the integral identity
\[ \int_{-1}^1 \sin(\pi m x/2) \sin (\pi k x) \ dx = \begin{cases} 1 & \text{ if } m \text{ is even and } k = m/2 \\ 0&  \text{ if } m \text{ is even and } k \neq m/2 \\ \frac{8k (-1)^k\sin(\pi m /2)}{\pi m^2 - 4\pi k^2 }& \text{ if } m \text{ is odd. } \end{cases}\]
For all $k,m \in \mathbb{Z}$ we have $\int_{-1}^1 \sin(\pi m x/2) \cos(\pi k x) \ dx = 0$. Combining this and the above with 
\[ b_m = \frac{1}{2} \int_{-1}^1 \sin(\pi m x /2) f(x) \ dx, \]
gives \eqref{bm}. To derive \eqref{Am} and \eqref{Bm} we need relations between the exponential Fourier coefficients and the sine-cosine Fourier coefficients. These are
\[ 2\hat{f}(m) = a_m -ib_m , \quad 2\hat{f}(-m) = a_m +ib_m , \quad A_m = \hat{M}(m) + \hat{M}(-m), \quad B_m = i(\hat{M}(m)-\hat{M}(-m)).\]
By Lemma~\ref{expfou} and the above identities we have 
\begin{align*}
A_m &= \hat{M}(m)+\hat{M}(-m) \\
& = \frac{4}{m \pi}\sin(k \pi /2)\left( \overline{\hat{f}(m) }+\overline{\hat{f}(-m) } \right)- 4 \left( |\hat{f}(m) |^2+|\hat{f}(-m) |^2\right)\\
& = \frac{4 \sin \left(m\pi /2 \right)}{m \pi}  a_m - 2(a_m^2+b_m^2).\end{align*}
Equation \eqref{Bm} is derived similarly. 
\end{proof}

The identities in Lemma~\ref{fourel} will be used to construct constraints in our later programs. In order to work with the infinite sums of \eqref{am} and \eqref{bm} we will bound the tails of the series using Parseval's identity and the Cauchy-Schwarz inequality.

\begin{lemma}\label{tailbounds} Let $T$ be a positive integer, and $f(x)$ be as in \eqref{fgm} with sine-cosine Fourier series \eqref{trigf1}. Then for all integers $1 \leq m < 2T$ we have

\[ \left|\frac{2m}{\pi} \sum_{k=T+1}^\infty \frac{(-1)^k \sin(\pi m/2)}{m^2-4k^2} c_k\right| \leq \frac{1}{4-m^2/T^2} \cdot \frac{2m}{\pi \sqrt{6T^3}},\]
and
\[ \left|\frac{4}{\pi} \sum_{k=T+1}^\infty \frac{k(-1)^k \sin(\pi m/2)}{m^2-4k^2} d_k\right| \leq \frac{1}{4-m^2/T^2}\cdot \frac{4}{\pi\sqrt{2T}}.  \]

\end{lemma}

\begin{proof} Let $\hat{f}$ denote the Fourier transform of $f(x)$ on $[-1,1]$ defined for $k \in \mathbb{Z}$ by  
\[ \hat{f}(k) =\frac{1}{2} \int_{-1}^1 e^{-i\pi k x}f(x) \ dx.\]
For $k \geq 1$ we have $\hat{f}(k) = (c_k-id_k)/2$, and $f(-k) =(c_k+id_k)/2$. By Parseval's identity and the fact $\hat{f}(0) = 1/2$ we obtain 
\begin{equation}\label{parseval} 1 \geq \int_{-1}^1 f^2(x) \ dx = 2 \sum_{k \in \mathbb{Z}} |\hat{f}(k)|^2 = \frac{1}{2} + \sum_{k=1}^\infty (c_k^2+d_k^2).\end{equation}
Fix an integer $1 \leq m < 2T$. By the triangle inequality
\[ \left| \sum_{k=T+1}^\infty \frac{(-1)^k \sin(\pi m/2)}{m^2-4k^2} c_k\right|   \leq \sum_{k=T+1}^\infty \frac{|c_k|}{4k^2-m^2} \leq \frac{1}{4-m^2/T^2} \sum_{k=T+1}^\infty \frac{|c_k|}{k^2 } .\]
By the Cauchy-Schwarz inequality and \eqref{parseval}
\[ \left(\sum_{k=T+1}^\infty  \frac{|c_k|}{k^2 } \right)^2 \leq \left(\sum_{k=T+1}^\infty  c_k^2 \right)\left(\sum_{k=T+1}^\infty  \frac{1}{k^4 } \right) \leq \frac{1}{2}\int_T^\infty \frac{1}{x^4} \ dx = \frac{1}{6T^3}. \]

Combining the two last above lines gives the first tail estimate. We proceed similarly for the second. By the triangle inequality
\[ \left| \sum_{k=T+1}^\infty \frac{k(-1)^k \sin(\pi m/2)}{m^2-4k^2} d_k\right|   \leq \sum_{k=T+1}^\infty \frac{k|d_k|}{4k^2-m^2} \leq \frac{1}{4-m^2/T^2} \sum_{k=T+1}^\infty \frac{|d_k|}{k } .\]
By the Cauchy-Schwarz inequality and \eqref{parseval}
\[ \left(\sum_{k=T+1}^\infty  \frac{|d_k|}{k} \right)^2 \leq \left(\sum_{k=T+1}^\infty  d_k^2 \right)\left(\sum_{k=T+1}^\infty  \frac{1}{k^2 } \right) \leq \frac{1}{2}\int_T^\infty \frac{1}{x^2} \ dx = \frac{1}{2T}. \]
Combining the two last above lines gives the second tail estimate.

\end{proof}

\subsection{Properties using average values of $M(x)$ on small intervals }

In our convex program, the Fourier coefficients $c_k,d_k$ described above will be variables. The other main type of variable will represent average values of $M(x)$ on small intervals. Throughout the remainder of this work $N$ will denote a large positive integer and we'll also define $L = 2/N$. For each $1 \leq j \leq N$ define 
\[ w_j = \frac{1}{L} \int_{(j-1)L}^{jL} M(x) \ dx\quad \text{and} \quad v_j = \frac{1}{L} \int_{-jL}^{-(j-1)L} M(x) \ dx.\]

Note that 

\begin{equation}\label{mmass} 1=  \int_{-2}^2 M(x) \ dx = \sum_{j=1}^N \int_{(j-1)L}^{jL} (M(x)+M(-x)) \ dx = L\sum_{j=1}^N (w_j+v_j).\end{equation}

In this subsection we will estimate the Fourier coefficients, the second moment, and the mean of $M(x)$ using the average values $\{w_j,v_j\}_{j=1}^N$. Let $R$ be a positive integer. For all $1 \leq j \leq N$ and $1 \leq m \leq R$, let $\alpha^+_{j,m}$ and $\alpha^-_{j,m}$ be upper and lower bounds of $\cos (\pi m x/2)$ on the interval $[(j-1)L,jL]$, i.e.
\begin{equation}\label{alphaarray} \alpha_{j,m}^+ \geq \max_{(j-1)L \leq x \leq jL} \cos (\pi m x/2) \quad \text{and}\quad \alpha_{j,m}^- \leq \min_{(j-1)L \leq x \leq jL} \cos (\pi m x/2).\end{equation}
Similarly, define $\beta^+_{j,m}$ and $\beta^-_{j,m}$ to be upper and lower bounds of $\sin (\pi m x/2)$ on the same intervals:
\begin{equation}\label{betaarray} \beta_{j,m}^+ \geq \max_{(j-1)L \leq x \leq jL} \sin (\pi m x/2) \quad \text{and}\quad \beta_{j,m}^- \leq \min_{(j-1)L \leq x \leq jL} \sin (\pi m x/2).\end{equation}

Next, we use the arrays $\{\alpha^+_{j,m},\alpha^-_{j,m}\}$ and $\{\beta^+_{j,m},\beta^-_{j,m}\}$ to give estimates on the Fourier coefficients of $M(x)$.

\begin{lemma}\label{trigintervallem} For all $1 \leq m \leq R$
\[\frac{L}{2}  \sum_{j=1}^N \alpha_{j,m}^-(w_j+v_j)  \leq A_m \leq \frac{L}{2} \sum_{j=1}^N \alpha_{j,m}^+(w_j+v_j). \]
\[ \frac{L}{2} \sum_{j=1}^N (\beta_{j,m}^-w_j-\beta_{j,m}^+v_j) \leq B_m \leq \frac{L}{2} \sum_{j=1}^N (\beta_{j,m}^+w_j-\beta_{j,m}^-v_j). \]
\end{lemma} 

\begin{proof} Let $m \geq 1$. We can break the integrals defining $A_r$ into a sum of integrals on the small intervals:
\begin{align} A_m &=\frac{1}{2}  \int_{-2}^2 \cos(\pi m x/2) M(x) \ dx \nonumber\\
& =  \frac{1}{2} \int_{-2}^2 \cos(\pi m x/2) (M(x) +M(-x))\ dx \nonumber\\
& = \frac{1}{2} \sum_{j=1}^N \int_{(j-1)L}^{jL}  \cos(\pi m x/2) (M(x) +M(-x))\ dx .  \label{2Arsum} \end{align}
By definition of $\alpha_{j,m}^+$, for all $1 \leq j \leq N$ we have 
\[ \int_{(j-1)L}^{jL}  \cos(\pi m x/2) (M(x) +M(-x))\ dx \leq \alpha_{j,m}^+ \int_{(j-1)L}^{jL} (M(x) +M(-x)) \ dx = L\alpha_{j,m}^+(w_j+v_j).\]
Substituting the above back into \eqref{2Arsum} gives the stated upper bound on $A_m$. The lower bound is similar. We can also break $B_m$ into the following sum of integrals.
\[ B_m = \frac{1}{2} \sum_{j=1}^N \int_{(j-1)L}^{jL}  \sin(\pi m x/2) (M(x) -M(-x))\ dx .\]
For all $1 \leq j \leq N$ and $(j-1)L \leq x \leq jL$ we have 
\[ \beta_{j,m}^- M(x) -\beta_{j,m}^+ M(-x)\leq \sin(\pi m x/2) (M(x) -M(-x)) \leq \beta_{j,m}^+ M(x) -\beta_{j,m}^- M(-x). \]
Substituting these estimates into the above expression for $B_m$ gives the stated upper and lower bounds on $B_m$.

\end{proof}

Define the mean of $f,g,M$ as follows. 
\[ E(f) = \int_{-1}^1 xf(x) \ dx,   \quad E(g) = \int_{-1}^1 xg(x) \ dx,   \quad E(M)= \int_{-2}^2 xM(x) \ dx. \]

We will use a bound on the second moment of $M(x)$. A similar bound is used in the work of Moser and Murdeshwar~\cite{M2}.

\begin{lemma}\label{centralM} For all $M$ as in \eqref{fgm} we have
\[ \int_{-2}^2 x^2M(x) \ dx = \frac{2}{3} + \frac{1}{2}E(M)^2.\]
\end{lemma}
\begin{proof} By changing the order of integration, we obtain 
\[ E(M) = \int_{-1}^1 f(t) \int_{-2}^2 (x+t)g(x+t) \ dx  dt- \int_{-1}^1 tf(t) \int_{-2}^2 g(x+t) \ dx  dt  = E(g)-E(f).\]
A similar change of order gives
\begin{align*} \int_{-2}^2 x^2M(x) \ dx &= \int_{-1}^1 f(t) \int_{-2}^2 x^2g(x+t) \ dxdt \\
& =\int_{-1}^1 f(t) \int_{-2}^2 ((x+t)^2-2t(x+t)+t^2)g(x+t) \ dxdt\\
&  = \int_{-1}^1 x^2 g(x) \ dx -2E(f)E(g) + \int_{-1}^1 t^2 f(t) \ dt \\
& =\int_{-1}^1 x^2(g(x)+f(x)) \ dx - 2E(f)E(g) = \frac{2}{3}-2E(f)E(g). \end{align*}

A simple calculation shows that $E(f) = -E(g)$ and so $E(f) = -E(M)/2$. Substituting these identities into the above gives the claimed result.

\end{proof}

Now we give an estimation of the mean and second moment of $M(x)$ using $\{w_j,v_j\}_{j=1}^N$.

\begin{lemma}\label{momboundslem} Let $h_1\leq h_2$ be real numbers such that $h_1 \leq E(M) \leq h_2$. Then
  \[ h_1\leq L^2 \sum_{j=1}^{N}  (jw_j-(j-1)v_j) \quad \text{and} \quad  L^2 \sum_{j=1}^{N}  ((j-1)w_j-jv_j) \leq h_2 ,\]
  and 
\[    2/3+h_1^2/2 \leq L^3 \sum_{j=1}^{N}  j^2 (w_j +v_j) \quad \text{and} \quad L^3 \sum_{j=1}^{N}  (j-1)^2 (w_j +v_j)\leq 2/3+h_2^2/2.\]

\end{lemma} 

\begin{proof} Once again we break down the appropriate integral into a sum of integrals. 
\[  E(M) = \int_{-2}^2 x M(x) \ dx = \sum_{j=1}^N  \int_{(j-1)L}^{jL} x(M(x)-M(-x)) \ dx .\]
For all $1 \leq j \leq N$ we can upper and lower bound the summand above.
\[ (j-1)L^2w_j -jL^2v_j \leq   \int_{(j-1)L}^{jL} x(M(x)-M(-x)) \ dx \leq jL^2w_j -(j-1)L^2v_j.\]
Substituting this back into the above expression for $E(M)$ and then applying $h_1 \leq E(M) \leq h_2$ gives the upper and lower bounds stated. Secondly for the second moment: 
\[ \int_{-2}^2 x^2  M(x) \ dx = \sum_{j=1}^N  \int_{(j-1)L}^{jL} x^2(M(x)+M(-x)) \ dx .\]
For all $1 \leq j \leq N$ we can again upper and lower bound the summand above.
\[ (j-1)^2L^3(w_j+v_j) \leq   \int_{(j-1)L}^{jL} x^2(M(x)-M(-x)) \ dx \leq j^2L^3(w_j+v_j).\]
Substituting this estimate into the above expression for the second moment gives 
\[ L^3 \sum_{j=1}^{N}  j^2 (w_j +v_j)\leq \int_{-2}^2 x^2  M(x) \ dx  \leq L^3 \sum_{j=1}^{N}  (j-1)^2 (w_j +v_j).\]
Applying Lemma~\ref{centralM} gives the two required inequalities.

\end{proof}


\section{Simplified linear program}\label{simpsec}

In this section we describe a simplified linear program version of our technique that is meant to convey the main idea of the more complex convex program in the following section. In this section we will derive a lower bound on $\| M \|_{\infty}$ under the additional assumption that $M$ is even, i.e. $M(x) = M(-x)$. The input of the linear program will be positive integers $N,R$ and $L = 2/N$. We will also explicitly choose values for the array $\{\alpha_{j,m}^-\}$ defined in \eqref{alphaarray}. Set 
\[\alpha_{j,m}^- = \cos(\pi mL(j-1/2)/2) - \pi m L/4,  \quad 1 \leq j \leq N, 1 \leq m \leq 2R.\]
Note that the above choice satisfies the definition \eqref{alphaarray} since the derivative of $\cos(\pi m x/2)$ is bounded in absolute value by $\pi m/2$. Our linear program is the following.

\begin{align}
\textsc{Input: } N,  R, L = 2/N \nonumber \\
\textsc{Variables: }& \Omega,w_1,\ldots,w_N  \nonumber \\
\textsc{Minimize: } &\Omega   \nonumber\\
\textsc{Subject to: } & 0 \leq w_j \leq \Omega; \quad 1 \leq j \leq N, \label{simplebound} \\
  & \sum_{j=1}^{N}w_j = N/4, \label{simplemass}\\
& \sum_{j=1}^{N} \alpha_{j,2m}^-w_j \leq 0; \quad 1 \leq m \leq R, \label{simplecos} \\
& L^3 \sum_{j=1}^{N}  (j-1)^2 w_j \leq 1/3.\label{simplemom}
\end{align}

\begin{prop} Let $N,R$ be arbitrary positive integers, and $M(x)$ be as in \eqref{fgm} with the additional property that $M(x)$ is even. If $\Omega^*$ is the optimum of the above program, then $\|M\|_{\infty} \geq \Omega^*$. 
\end{prop} 

\begin{proof} Our strategy is to show that any even $M(x)$ gives a feasible assignment of variables such that $\Omega \leq \| M \|_{\infty}$. Let $M(x)$ be even, fix positive integers $N,R$. Put $\Omega = \|M \|_{\infty}$ and $w_j = \frac{1}{L} \int_{(j-1)L}^{jL} M(x) \ dx$. Clearly this assignment of variables satisfies the constraints in \eqref{simplebound}. Since $M(x)$ is even, we have
\[ \frac{1}{2} = \int_{0}^2 M(x) \ dx = \sum_{j=1}^N \int_{(j-1)L}^{jL} M(x) \ dx = L \sum_{j=1}^N w_j,\]
satisfying constraint \eqref{simplemass}. By Lemma~\ref{trigintervallem} and \eqref{Am} of Lemma~\ref{fourel}, for all $1 \leq m \leq R$ we have 
\[ \frac{L}{2} \sum_{j=1}^N \alpha_{j,2m}^- w_j \leq A_{2m} \leq 0 ,\]
satisfying constraint \eqref{simplecos}. Lastly, from Lemma~\ref{centralM} we have 
\[ L^3 \sum_{j=1}^{N}  (j-1)^2 w_j \leq \int_{0}^2 x^2 M(x) \ dx = 1/3,\]
satisfying constraint \eqref{simplemom}.

\end{proof} 

The optimum of this linear program increases with $N$ and $R$. A linear program with hundreds of constraints and variables can be solved exactly. For our linear program, it is best to choose $N \geq 2000$. At this scale of variables and constraints, exact-solvers are computationally slow, but numerical solvers are still very fast. The disadvantage of numerical solvers is that additional post-processing steps must be taken to ensure the reported results are correct. Our strategy will be to find a feasible point in the dual program, which will give a lower bound on the minimum to the primal problem. To guarantee our dual solution is truly feasible, we will check that all inequalities that define the dual space are strictly satisfied by a margin that exceeds the worst-case-scenario for floating-point rounding errors.

Choosing $N = 80000$ and $R = 20$ and running the dual program in CPLEX returns a feasible point in the dual space with objective 0.375169005340707. We have verified that all constraints are satisfied by a margin exceeding the worst-case floating point error accumulation. Therefore $\|M \|_{\infty} \geq 0.375$ for all even $M$. By increasing $N$ and $R$ this bound can be improved a little, but it seems that the limit of this approach is less than $0.3755$. In the next section we will introduce nonlinear constraints that will give a substantial improvement.



\section{Full convex program}\label{fullconsec}

In this section we give and discuss our main program that will be used to prove Theorem~\ref{T1}. The convex program we will use includes the ideas of the previous linear program, but will also add constraints coming from \eqref{Bm} which will be important since we no longer assume $M(x)$ is even. We will also add quadratic constraints coming from \eqref{Am}. There will be several inputs to our convex program. Let $N,R$ denote positive integers, their role will be roughly the same as in the simplified linear program. A positive integer $T$ will also be an input, its role is as in Lemma~\ref{tailbounds} and determines when we truncate certain infinite sums. Finally our last inputs will be real numbers $h_1<h_2,p_1<p_2,q_1<q_2$. These inputs will act as bounds on the mean of $M(x)$, first cosine, and first sine Fourier coefficient of $f(x)$. We will also explicitly choose values for the arrays $\{\alpha_{j,m}^-,\alpha_{j,m}^+\}$ and $\{\beta_{j,m}^-,\beta_{j,m}^+\}$ defined in \eqref{alphaarray} and \eqref{betaarray}. Set 
\[\alpha_{j,m}^- = \cos(\pi mL(j-1/2)/2) - \pi m L/4,  \quad \alpha_{j,m}^+ = \cos(\pi mL(j-1/2)/2) + \pi m L/4,\]
\[\beta_{j,m}^- = \sin(\pi mL(j-1/2)/2) - \pi m L/4, \quad \beta_{j,m}^+ = \sin(\pi mL(j-1/2)/2) +\pi m L/4,\]
for all $1 \leq j \leq N$ and $1 \leq m \leq 2R$. As above, these choices satisfy the definitions of \eqref{alphaarray} and \eqref{betaarray} since the derivatives of $\cos(\pi m x/2)$ and $\sin(\pi m x/2)$ are bounded by $\pi m/2$. Our full convex program is the following.


\begin{align}
\textsc{Input: } & N,L=2/N,T,R,h_1,h_2,p_1,p_2,q_1,q_2 \nonumber \\
\textsc{Variables: }& \Omega, \{w_j,v_j\}_{j=1}^N,\{c_k,d_k\}_{k=1}^T,\{\epsilon_{2m-1},\delta_{2m-1}\}_{j=1}^R,\nonumber \\
\textsc{Variable expressions: } &\text{For } 1 \leq m \leq 2R:\nonumber\\
& a_m = \begin{cases} \frac{1}{2}c_{m/2} & \text{if }m \text{ is even} \\ 
\epsilon_m+\frac{2m\sin(\pi m/2)}{\pi} \left( \frac{1}{2m^2}+ \sum_{k=1}^T \frac{(-1)^k }{m^2-4k^2} c_k \right) & \text{if }m \text{ is odd} \end{cases}  \nonumber \\
&b_m = \begin{cases} \frac{1}{2}d_{m/2} & \text{if }m \text{ is even} \\ 
\delta_m+ \frac{4}{\pi} \sum_{k=1}^T \frac{k(-1)^k \sin(\pi m/2)}{m^2-4k^2} d_k & \text{if }m \text{ is odd} \end{cases} \nonumber \\
\textsc{Minimize: } &\Omega  \nonumber \\
\textsc{Subject to: }  &0 \leq w_j ,v_j \leq \Omega \leq 1; \quad 1 \leq j \leq N,\label{fullbound} \\
   L\sum_{j=1}^{N} &(w_j+v_j) = 1, \label{fullmass} \\
   L^2 \sum_{j=1}^{N}  &(jw_j-(j-1)v_j) \geq h_1 \label{fullav}\\
    L^3 \sum_{j=1}^{N} & (j-1)^2 (w_j +v_j) \leq 2/3+h_2^2/2, \label{fullmom} \\
     \frac{L}{2}\sum_{j=1}^{N} \alpha_{j,m}^-(w_j+v_j) & \leq  \frac{4 \sin \left(m\pi /2 \right)}{m \pi}  a_m - 2(a_m^2+b_m^2); \quad 1 \leq m \leq 2R,\label{fullcos} \\
   \frac{L}{2}\sum_{j=1}^N (\beta_{j,m}^-w_j-\beta_{j,m}^+v_j) &\leq -\frac{8}{m \pi}\sin\left( m\pi /2 \right) b_m ; \quad 1 \leq m \leq 2R, \label{fullsinep} \\
     \frac{L}{2}\sum_{j=1}^N (\beta_{j,m}^+w_j-\beta_{j,m}^-v_j) & \geq -\frac{8}{m \pi}\sin\left( m\pi /2 \right) b_m; \quad 1 \leq m \leq 2R,\label{fullsinem}\\
 | \epsilon_{2m-1} | & \leq \frac{1}{4-m^2/T^2} \cdot \frac{2m}{\pi \sqrt{6T^3}} \quad 1 \leq m \leq R, \label{fullep}\\
  | \delta_{2m-1} | & \leq \frac{1}{4-m^2/T^2}\cdot \frac{4}{\pi\sqrt{2T}} \quad 1 \leq m \leq R,\label{fulldel}\\
  |c_k |, & \ |d_k|  \leq \frac{2}{\pi} \quad 1 \leq k \leq T, \label{fullckdk}\\
&  \sum_{k=1}^T  (c_k^2+d_k^2) \leq 1/2 ,\label{fullpars} \\
   & p_1 \leq c_1 \leq p_2, \quad q_1 \leq d_1 \leq q_2,\label{fullc1d1} \\
      \frac{L}{2}&\sum_{j=1}^{N} \alpha_{j,2}^+(w_j+v_j) \geq - \frac{1}{2}(p_2^2+\max\{q_1^2,q_2^2\}).\label{fullA1} 
\end{align}


Our next proposition shows how the input and output of the convex program relate to an $f(x),M(x)$ pair of functions.

\begin{prop}\label{convexprop} Let $N,T,R$ be arbitrary positive integers, and $h_1,h_2,p_1,p_2,q_1,q_2$ be real numbers. Let $\Omega^*$ be the optimum of the above program with this choice of input. Suppose that $f(x),M(x)$ is as in \eqref{fgm} and satisfies 
\begin{enumerate}[(i)]
\item $0 \leq h_1\leq E(M) \leq h_2$,
\item $0 \leq p_1 \leq \int_{-1}^1 \cos(\pi x) f(x) \ dx \leq p_2$,
\item $q_1 \leq \int_{-1}^1 \sin(\pi x) f(x) \ dx \leq q_2$.
\end{enumerate}
Then $\|M\|_{\infty} \geq \Omega^*$.
\end{prop}

\begin{proof} Our strategy is to show that any $f(x),M(x)$ satisfying the above hypotheses gives a feasible assignment of variables such that $\Omega \leq \| M \|_{\infty}$. Let $f(x),M(x)$ be as in \eqref{fgm} such that (i),(ii),(iii) above are also satisfied. We will show that the following choice of variables is feasible. Set the variable $\Omega = \|M\|_{\infty}$. For each $1 \leq j \leq N$ set 
\[ w_j = \frac{1}{L} \int_{(j-1)L}^{jL} M(x) \ dx\quad \text{and} \quad v_j = \frac{1}{L} \int_{-jL}^{-(j-1)L} M(x) \ dx.\]
For all $k \geq 1$ set
\begin{equation}\label{ckdkdef} c_k = \int_{-1}^1 \cos(\pi kx ) f(x) \ dx \quad \text{and} \quad d_k =  \int_{-1}^1 \sin(\pi k x) f(x) \ dx.\end{equation}
The variables of the program $c_k,d_k$ only run from the indices $1 \leq k \leq T$, but we define $c_k,d_k$ for $k \geq T+1$ to aid in our definition of the last two types of variable. For each $1 \leq m \leq R$ set 
\[ \epsilon_{2m-1} = \frac{2m (-1)^{m+1}}{\pi} \sum_{k=T+1}^\infty \frac{(-1)^k}{(2m-1)^2-4k^2} c_k, \quad \text{and} \]
\[ \delta_{2m-1} = \frac{4(-1)^{m+1}}{\pi} \sum_{k=T+1}^\infty \frac{k(-1)^k }{(2m-1)^2-4k^2} d_k.\]

We'll now show that this assignment of variables is feasible. Constraint \eqref{fullbound} follows immediately from the definitions; the average of $M(x)$ on an interval cannot exceed the maximum. The other constraints are satisfied for reasons contained in the lemmas of Section~\ref{setup}, or hypotheses (i),(ii),(iii). The dependency of constraints on corresponding lemma(s) is outlined in Table~\ref{contable}. 

\begin{itemize}
\item Constraint \eqref{fullmass} follows from \eqref{mmass}, i.e. because $\int_{-1}^1 M(x) \ dx = 1$. 
\item Constraints \eqref{fullav} and \eqref{fullmom} follow from Lemma~\ref{momboundslem} and hypothesis (i). The constraints represent bounds on the mean and second moment of $M(x)$, respectively.
\item By Lemma~\ref{fourel}, the righthand side of constraint \eqref{fullcos} is the $m^{th}$ cosine Fourier coefficient of $M(x)$. By Lemma~\ref{trigintervallem}, the lefthand side of \eqref{fullcos} is a lower bound on this Fourier coefficient.
\item By Lemma~\ref{fourel}, the righthand side of constraints \eqref{fullsinep} and \eqref{fullsinem} are the $m^{th}$ sine Fourier coefficient of $M(x)$. By Lemma~\ref{trigintervallem}, the righthand sides of \eqref{fullsinep} and \eqref{fullsinem} are lower and upper bounds, respectively, on this Fourier coefficient.
\item Constraints \eqref{fullep} and \eqref{fulldel} follow from Lemma~\ref{tailbounds}.
\item Constraint \eqref{fullckdk} follows from \eqref{ckdkdef} since $f(x) \in [0,1]$ and $\|f\|_1 = 1$. 
\item Constraint \eqref{fullpars} follows from Parseval's identity, in particular equation \eqref{parseval}.
\item The left constraint of \eqref{fullc1d1} follows from hypothesis (ii), since the integer in (ii) is precisely $c_1$. Similarly, the right constraint of \eqref{fullc1d1} follows from hypothesis (iii).

\item By Lemma~\ref{fourel}, in particular \eqref{am} and \eqref{Am}, the second cosine Fourier coefficient of $M(x)$ is 
\[ A_2  = \frac{1}{2} \int_{-2}^2 \cos(\pi x) \ dx = -\frac{1}{2}(c_1^2+d_1^2).\]
From constraint \eqref{fullc1d1} we see that $-\frac{1}{2}(p_2^2+\max\{q_1^2,q_2^2\})$ is a lower bound on the righthand side above. Also, from Lemma~\ref{trigintervallem} we see that the lefthand side of constraint \eqref{fullA1} is an upper bound on the lefthand side above. 
\end{itemize}

\end{proof}

\renewcommand{\arraystretch}{1.3}

\begin{table}[h!] \centering
 \begin{tabular}{| c |c |} 
 \hline
          Constraint & Follows from \\ 
 \hline
 \hline
 \eqref{fullmass} & \eqref{mmass} \\ 
 \hline
 \eqref{fullav}, \eqref{fullmom}  & Lemma~\ref{momboundslem} and (i) \\ 
 \hline
   \eqref{fullcos} & Lemma~\ref{fourel}: \eqref{Am} and Lemma~\ref{trigintervallem} \\ 
   \hline
    \eqref{fullsinep}, \eqref{fullsinem} & Lemma~\ref{fourel}: \eqref{Bm} and Lemma~\ref{trigintervallem} \\ 
 \hline
  \eqref{fullep}, \eqref{fulldel}& Lemma~\ref{tailbounds} \\ 
   \hline
     \eqref{fullckdk}& \eqref{ckdkdef} \\ 
   \hline
   \eqref{fullpars} & \eqref{parseval} \\ 
 \hline
   \eqref{fullc1d1} & (ii) and (iii) \\ 
 \hline
    \eqref{fullA1} & Lemma~\ref{fourel}: \eqref{Am}, Lemma~\ref{trigintervallem}, (ii) and (iii) \\ 
 \hline
\end{tabular}
\caption{Constraint feasibility}
\label{contable}
\end{table}

For the remainder of this paper, let $f^\ast,M^\ast$ be an optimal pair in the sense that $\|M^\ast\|_{\infty}=~\mu$. Also denote the first sine and cosine Fourier coefficients of $f^\ast$ by $c_1^\ast$ and $d_1^\ast$, i.e.  
\[ c_1^\ast = \int_{-1}^1 \cos(\pi x) f^\ast (x) \ dx , \quad \text{\and} \quad d_1^\ast = \int_{-1}^1 \sin(\pi x) f^\ast (x) \ dx .\]
Note that $f^\ast(x)$, $f^\ast(-x)$, and $1-f^\ast(x)$ are all also optimal. This means that without loss of generality, we can assume that $E(M^\ast) \geq 0$ and $c_1^\ast \geq 0$. It's also easy to see that $E(M^*) \leq 2$ and $|c_1^*|,|d_1^*| \leq 1$. If we choose the inputs 
\begin{equation}\label{valinter} (h_1,h_2) = (0,2), \quad (p_1,p_2) = (0,1), \quad (q_1,q_2) = (-1,1), \end{equation}
then the optimum of our convex program with this input gives a lower bound on $\mu$. However, the optimum ends up close to 0.25 for any choice of $N,T,R$. In order to use our program to give a good bound on $\mu$, we will need a `divide and conquer' strategy of breaking up the valid ranges of parameters shown in \eqref{valinter} into small chunks. The minimum of all the optimums over these smaller intervals will be our lower bound on $\mu$.

\subsection{Dual program output}

As was the case with our simpler linear program in Section~\ref{simpsec}, it is best to choose the parameter $N$ to be fairly large, i.e. at least 2000. This rules out the use of exact solver algorithms, but numerical solvers are still fast enough for this purpose. To guarantee the correctness of our output, we will find feasible points in the interior of the dual program space, sufficiently far from the boundary to compensate for any floating-point arithmetic errors. The dual of a convex quadratically constrained problem is easiest to calculate by first reformulating the primal program as a second order cone program (SOCP). The dual of a SOCP is relatively simple to write down. We have placed the details of these steps in two appendices. In Appendix II we explicitly write down the dual program and discuss our post-processing verification step to ensure that even if worst-case floating point arithmetic errors occur, the assignments of dual variables used are definitely feasible. In Table~\ref{dprogout} we display the value of the objective function of the dual program for a verified feasible point in the dual space for several choices of input. Note that we have used the `divide and conquer' strategy mentioned above to narrow down the location of $E(M^\ast),c_1^\ast$ and $d_1^\ast$.

\begin{table}[h!] \centering
\begin{scriptsize}

 \begin{tabular}{| c |c |} 
 \hline
          \makecell{Parameter assignment \\ $N,T ,R= 10000,4000,10$ }  & Optimum lower bound \\
 \hline
 \hline
 $(h_1,h_2) ,(p_1,p_2),(q_1,q_2 )=(0.75,2),(0,1),(-1,1) $ & 0.38\\
 \hline
 $(h_1,h_2) ,(p_1,p_2),(q_1,q_2 )=(0.4,0.75),(0,1),(-1,1) $ & 0.38\\
 \hline
 $(h_1,h_2) ,(p_1,p_2),(q_1,q_2 )=(0.2,0.4),(0,1),(-1,1) $ & 0.38\\
 \hline
 $(h_1,h_2) ,(p_1,p_2),(q_1,q_2 )=(0.1,0.2),(0,1),(-1,1) $ & 0.38\\
 \hline
 $(h_1,h_2) ,(p_1,p_2),(q_1,q_2 )=(0.08,0.1),(0,1),(-1,1) $ & 0.38\\
 \hline
 $(h_1,h_2) ,(p_1,p_2),(q_1,q_2 )=(0,0.08),(0,1),(-1,-0.05) $ & 0.38\\
 \hline
 $(h_1,h_2) ,(p_1,p_2),(q_1,q_2 )=(0,0.08),(0,1),(-0.05,-0.025) $ & 0.38\\
 \hline
  $(h_1,h_2) ,(p_1,p_2),(q_1,q_2 )=(0,0.08),(0,1),(0.05,1) $ & 0.38\\
 \hline
 $(h_1,h_2) ,(p_1,p_2),(q_1,q_2 )=(0,0.08),(0,1),(0.025,0.05) $ & 0.38\\
 \hline
 $(h_1,h_2) ,(p_1,p_2),(q_1,q_2 )=(0,0.08),(0,0.25),(-0.025,0.025) $ & 0.38\\
\hline
 $(h_1,h_2) ,(p_1,p_2),(q_1,q_2 )=(0,0.08),(0.25,0.3),(-0.025,0.025) $ & 0.38\\
\hline
$(h_1,h_2) ,(p_1,p_2),(q_1,q_2 )=(0,0.08),(0.3,0.33),(-0.025,0.025) $ & 0.38\\
\hline
$(h_1,h_2) ,(p_1,p_2),(q_1,q_2 )=(0,0.08),(0.5,1),(-0.025,0.025) $ & 0.38\\
\hline
$(h_1,h_2) ,(p_1,p_2),(q_1,q_2 )=(0,0.08),(0.45,0.5),(-0.025,0.025) $ & 0.38\\
\hline
$(h_1,h_2) ,(p_1,p_2),(q_1,q_2 )=(0.06,0.08),(0.33,0.45),(-0.025,0.025) $ & 0.38\\
\hline
$N,T,R = 20000,5000,10$ & \\
\hline
\hline
$(h_1,h_2) ,(p_1,p_2),(q_1,q_2 )=(0.0,0.06),(0.33,0.45),(-0.025,-0.02) $ &0.38 \\
\hline
$(h_1,h_2) ,(p_1,p_2),(q_1,q_2 )=(0.0,0.06),(0.33,0.45),(0.02,0.025) $ &0.38 \\
\hline
$(h_1,h_2) ,(p_1,p_2),(q_1,q_2 )=(0.0,0.06),(0.33,0.35),(-0.02,0.02) $ &0.37925 \\
\hline

\end{tabular}
\end{scriptsize}

\caption{Dual program output I}
\label{dprogout}
\end{table}

The data of the first line of Table~\ref{dprogout} shows that either $\mu \geq 0.38$ or $E(M^\ast) \leq 0.75$. Combining all data from Table~\ref{dprogout} shows that either $\mu \geq 0.37925$ or 
\begin{equation}\label{lastregion} 0 \leq E(M^\ast) \leq 0.06, \quad 0.35 \leq c_1^\ast \leq 0.45, \quad \text{and } -0.02 \leq d_1^\ast \leq 0.02.\end{equation}

We can proceed in the same way by dividing the above intervals into small pieces and finding feasible points for each of the corresponding inputs. The drawback to this approach is that in order to achieve a lower bound close to 0.379, many hundreds of feasible points need to be computed and verified. Instead, we can use the observation that the parameters $h_1,h_2,p_1,p_2$ are not a part of the constraints in the dual program, i.e. they only affect the objective function. As a result, for any choice of input $N,T,R,h_1,h_2,p_1,p_2,q_1,q_2$ and a feasible assignment of dual variables, we have a lower bound on the optimum of the dual program for any choice of $h_1,h_2,p_1,p_2$ by reusing the variable assignment, and simply re-computing the objective function with the changed values of $h_1,h_2,p_1,p_2$. A precise statement, proof, and example of this idea is given in Appendix II. We will apply this technique by first calculating a feasible point in the dual with input such that $h = h_1=h_2$ and $p = p_1 = p_2$. Next we determine by how much can $(h,p)$ change while keeping the rest of the feasible point data fixed so that the objective function is always at least 0.379005. Since the dependence of the objective function on $h$ and $p$ is quadratic, the region for which the objective is at least 0.379005 will be an ellipse on $(h,p)$-axes. We compute seven feasible points such that the union of the ellipses described above covers the remaining region described in \eqref{lastregion}. The input of these feasible points is described in Table~\ref{dprogout2}. The covering ellipses are shown in Figure~\ref{ellipses}.



\begin{table}[h!] \centering
\begin{footnotesize}

 \begin{tabular}{|c| c | c| } 
 \hline
       Label/Colour in Figure~\ref{ellipses} &  Initial parameter assignment & Initial objective value  \\
 \hline
 1/Green & $h = 0.015$, $p = 0.381$ & 0.37905 \\
 \hline
  2/Blue &   $h= 0.015$, $p = 0.385$ & 0.37905 \\
 \hline
  3/Red & $h= 0.02$, $p = 0.375$ & 0.37905 \\
 \hline
  4/Purple &  $h = 0.004$, $p = 0.3875$ & 0.37905 \\ 
 \hline
  5/Light green & $h = 0$, $p = 0.4$ & 0.3791 \\
 \hline
  6/Orange & $h = 0$, $p = 0.381$  & 0.3791\\
 \hline
  7/Light blue & $h = 0.03$, $p = 0.375$  & 0.3794 \\
 \hline

\end{tabular}
\end{footnotesize}

\caption{Dual program output II \\ $N,T ,R= 25000,7000,10$, and $(q_1,q_2) = (-0.02,0.02)$}
\label{dprogout2}
\end{table}

\begin{figure}[h!]
  \centering
    \includegraphics[width=0.5\textwidth]{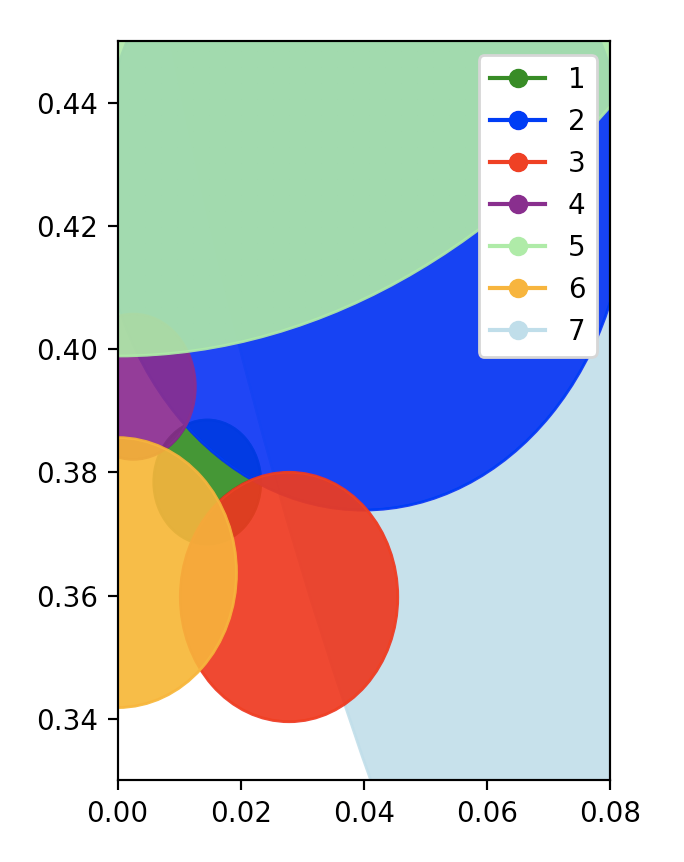}
    \caption{Lower bound of 0.379005  }
    \label{ellipses}
\end{figure}

The full variable assignments of all feasible solutions used in this section are available upon request to the author. We conclude that the Theorem~\ref{T1} bound $\mu \geq 0.379005$ is now verified.

\section{Concluding remarks} 


We expect that by computing more feasible solutions, using larger values of $N,T,R$, and using more accurate floating point calculations, that our lower bound of 0.379005 can be improved. Using the input $N,T,R = 25000,7000,10$ and 
\[ h_1 = h_2 = 0.015, \quad p_1 = p_2 = 0.381, \quad -q_1 = q_2 = 0.02,\]
we were able to find a feasible point with objective 0.37905, but not with objective 0.3791. This seems to indicate that the limit of this method is not much larger than 0.379. We also expect that the optimal function $f^\ast(x)$ is even. A proof of this would also improve the lower bound on $\mu$, since we could always take $h_1 = h_2 = q_1 = q_2 = 0$ as part of the input in our convex program. 


\vspace{1cm}
\textbf{Acknowledgements.} The author thanks J\'ozsef Solymosi and Josh Zahl for careful reading of this work and providing valuable suggestions.

\clearpage

\section{Appendix I: Primal program as a SOCP}

In order to construct the dual program of the convex program in Section~\ref{fullconsec} we first reformulate this primal program as a second order cone program (SOCP). In Appendix II, we reuse the data involved in the primal program to construct the dual SOCP. The input of both the primal and the dual is the same as the input of the convex program in Section~\ref{fullconsec}. The constraints in the primal SOCP are the exact same as the constraints in Section~\ref{fullconsec}. The only difference is notation of variables and that the quadratic constraints must be rearranged to become second order cone constraints. Throughout, it will be convenient to use square brackets to describe individual elements of vectors and arrays. For example $A[i,j] = A_{ij}$ denotes the $i,j^{th}$ entry. 
\vspace{0.15cm}

\noindent\textsc{Input: }  $N,L=2/N,T,R, h_1,h_2,p_1,p_2,q_1,q_2$ 
\vspace{0.15cm}

\noindent\textsc{Variables: } A $(2N+2T+2R+1) \times 1$ vector variable $X$. The entries of $X$ correspond to the variables in the original convex program. This correspondence is described in the following table of variable expressions. The entry $X_0$ corresponds to variable $\Omega$ in the original program. 

\vspace{0.05cm} 

\noindent\textsc{Variable expressions and data: } Variable expressions and data in the primal and dual (in Appendix II) will be arranged in arrays and vectors. See Table~\ref{socpex} for the expressions and Table~\ref{socpdata} for the data.  In the ``Description'' column we will indicate how the array or vector is indexed. We use the notation $[a,b]$ to indicate the set of integers between $a$ and $b$ inclusively, and $[1,a] = [a]$. For example, an $[a] \times [2,b]$ array $A$ will have entries $A[i,j]$ for $1 \leq i \leq a$ and $2 \leq j \leq b$. 

\vspace{0.15cm}

\noindent\textsc{Objective: minimize} $ \Phi^TX = X_0$ 

\vspace{0.15cm}

\noindent\textsc{Constraints: }  
\begin{scriptsize}
\begin{align*}
\big\| \text{A\textsubscript{coscone}}[m]X + \text{b\textsubscript{coscone}}[m] \big\|_2 &\leq \text{c$^T$\textsubscript{coscone}}[m]X + \text{d\textsubscript{coscone}}[m], \text{ for all } 1 \leq m \leq 2R \\
\big\| \text{A\textsubscript{par}}X  \big\|_2 &\leq  \text{d\textsubscript{par}},  \\
 \text{c$^T$\textsubscript{obnd}}X+\text{d\textsubscript{obnd}} &\geq 0, \\
 \text{c$^T$\textsubscript{wbnd}}[i,j]X &\geq 0,  \text{ for all } 1 \leq i \leq 2, \  1 \leq j \leq N,\\
  \text{c$^T$\textsubscript{vbnd}}[i,j]X &\geq 0,  \text{ for all } 1 \leq i \leq 2, \  1 \leq j \leq N,\\
 \text{c$^T$\textsubscript{sum}}[i]X   + \text{d\textsubscript{sum}}[i]&\geq 0,  \text{ for all } 1 \leq i \leq 2, \\
\text{c$^T$\textsubscript{mean}}X   + \text{d\textsubscript{mean}}&\geq 0,  \\
\text{c$^T$\textsubscript{mome}}X   + \text{d\textsubscript{mome}}&\geq 0,  \\
 \text{c$^T$\textsubscript{sin-lower}}[m]X &\geq 0,  \text{ for all } 1 \leq m \leq 2R \\
\text{c$^T$\textsubscript{sin-upper}}[m]X&\geq 0,    \text{ for all } 1 \leq m \leq 2R \\
\text{c$^T$\textsubscript{ckbnd}}[i,k]X   + \text{d\textsubscript{ckbnd}}[i,k] &\geq 0, \text{ for all } 1 \leq i \leq 2, \ 1 \leq k \leq T, \\
\text{c$^T$\textsubscript{dkbnd}}[i,k]X   + \text{d\textsubscript{dkbnd}}[i,k] &\geq 0, \text{ for all } 1 \leq i \leq 2, \ 1 \leq k \leq T, \\
\text{c$^T$\textsubscript{c1bnd}}[i]X   + \text{d\textsubscript{c1bnd}}[i] &\geq 0, \text{ for all } 1 \leq i \leq 2, \\
\text{c$^T$\textsubscript{d1bnd}}[i]X   + \text{d\textsubscript{d1bnd}}[i] &\geq 0, \text{ for all } 1 \leq i \leq 2, \\
\text{c$^T$\textsubscript{ep}}[i,m]X   + \text{d\textsubscript{ep}}[i,m] &\geq 0, \text{ for all } 1 \leq i \leq 2, \ 1 \leq m \leq R,  \\
\text{c$^T$\textsubscript{del}}[i,m]X   + \text{d\textsubscript{del}}[i,m] &\geq 0, \text{ for all } 1 \leq i \leq 2, \ 1 \leq m \leq R, \\
\text{c$^T$\textsubscript{cosup}}X   + \text{d\textsubscript{cosup}} &\geq 0.  \end{align*}

\end{scriptsize}
The set of constraints above is precisely the same as our program in Section~\ref{fullconsec}. For example, if both sides of the first constraint above are squared and simplified, we will obtain the constraint from the original convex program \eqref{fullcos}.

\renewcommand{\arraystretch}{1.5}

\begin{table}[!h] \centering
 \begin{tabular}{| c |c |c|} 
 \hline
          Name(s) & Description & Defining property \\ 
 \hline
 \hline

    $w,v$ &  $[N]$-expression vector  & $w_j = X_j$, $ v_j=X_{N+j} $ for $1 \leq j \leq N$ \\ 
   \hline
    $c,d$ &  $[T]$-expression vector  & $c_k = X_{2N+k}$, $ d_k=X_{2N+T+k} $ for $1 \leq k \leq T$ \\ 
   \hline
    $\epsilon,\delta$ &  $[R]$-expression vector& $\epsilon_m = X_{2N+2T+m}$, $ \delta_m=X_{2N+2T+R+m} $ for $1 \leq m \leq R$ \\ 
   \hline
 
  $a_m^*$ & $[2R]$-expression vector& $a_m^*=\begin{cases} \frac{1}{2}c_{m/2} & \text{if }m \text{ is even} \\ 
\epsilon_m+\frac{2m\sin(\pi m/2)}{\pi}  \sum_{k=1}^T \frac{(-1)^k }{m^2-4k^2} c_k & \text{if }m \text{ is odd} \end{cases} $ \\ 
    \hline
      $b_m$ &$[2R]$-expression vector & $ b_m=\begin{cases} \frac{1}{2}d_{m/2} & \text{if }m \text{ is even} \\ 
\delta_m+ \frac{4}{\pi} \sum_{k=1}^T \frac{k(-1)^k \sin(\pi m/2)}{m^2-4k^2} d_k & \text{if }m \text{ is odd} \end{cases} $ \\ 
    \hline
  Fcos\textsubscript{LB} & $[2R]$-expression vector &Fcos\textsubscript{LB}$[m]= \frac{L}{2}\sum_{j=1}^{N} \alpha_{j,m}^-(w_j+v_j)$ \\ 
 \hline
   Fcos\textsubscript{UB} & $[2R]$-expression vector & Fcos\textsubscript{UB}$[m]= \frac{L}{2}\sum_{j=1}^{N} \alpha_{j,m}^+(w_j+v_j)$ \\ 
 \hline
   Fsin\textsubscript{LB} & $[2R]$-expression vector &Fsin\textsubscript{UB}$[m]= \frac{L}{2}\sum_{j=1}^N (\beta_{j,m}^-w_j-\beta_{j,m}^+v_j)$ \\ 
 \hline
   Fsin\textsubscript{UB} & $[2R]$-expression vector &Fsin\textsubscript{UB}$[m]=  \frac{L}{2}\sum_{j=1}^N (\beta_{j,m}^+w_j-\beta_{j,m}^-v_j)$ \\ 
 \hline

\end{tabular}
\caption{Variable expressions for SOCP}
\label{socpex}
\end{table}

Remark: we will continue to use the values assigned to $\alpha_{j,m}$, $\beta_{j,m}$ in Section~\ref{fullconsec}.

\renewcommand{\arraystretch}{1.5}

\begin{table}[!h] \centering
\scriptsize
 \begin{tabular}{| c |c |c|}

 \hline
          Name& Description & Defining property \\ 
 \hline
 \hline

        $E$ & Integer & $E = 2N+2T+2R$ \\ 
   \hline
          $\Phi$ & $[0,E]$-vector & $\Phi_0 = 1$, $\Phi_j = 0$ if $1 \leq j \leq E$ \\ 
   \hline
\makecell{A\textsubscript{coscone}}  & $[2R] \times [3] \times [0,E]$ array & \makecell{A\textsubscript{coscone}$[m,1] X = \frac{\sin(m\pi /2)a^*_m}{m\pi} -$Fcos\textsubscript{LB}$[m] /2$\\ A\textsubscript{coscone}$[m,2] X =a^*_m$  \\ A\textsubscript{coscone}$[m,3] X =b_m $} \\ 
 \hline
  b\textsubscript{coscone} & $[2R] \times [3]$ array &  \makecell{ b\textsubscript{coscone}$[m,1] =-1/2+ 2\sin^2(\pi m/2)/(m\pi)^2 $\\ b\textsubscript{coscone} $[m,2] =\sin(\pi m/2)/(m\pi) $ \\ b\textsubscript{coscone} $[m,3] =0$ }\\ 
 \hline
   c\textsubscript{coscone} & $[2R] \times [0,E]$ array &  c$^T$\textsubscript{coscone}$[m]X =\frac{\sin(m\pi /2)a^*_m}{m\pi} -$Fcos\textsubscript{LB}$[m] /2$  \\ 
 \hline
   d\textsubscript{coscone} & $[2R]$-vector & d\textsubscript{coscone} $ [m] = 1/2+ 2\sin^2(\pi m/2)/(m\pi)^2$ \\ 
 \hline
 A\textsubscript{par} & $[2T]  \times [0,E]$ array & \makecell{A\textsubscript{par}$[k]X = c_k$, and \\ A\textsubscript{par}$[T+k]X= d_k$ for all $1 \leq k \leq T$} \\ 
 \hline
  d\textsubscript{par} & Real &d\textsubscript{par} $= 1/\sqrt{2}$ \\ 
 \hline
  \hline
     c\textsubscript{obnd}& $[0,E]$-vector &  c$^T$\textsubscript{obnd}$X= -\Omega$  \\ 
 \hline
    d\textsubscript{obnd}& Real &  d\textsubscript{obnd}$=1$\\
 \hline
   c\textsubscript{wbnd}& $[2] \times [N] \times [0,E]$ array &   \makecell{  c$^T$\textsubscript{wbnd}$[1,j]X = \Omega-w_j$, for all $1 \leq j \leq N$, and \\  c\textsubscript{wbnd}$[2,j]X = w_j$} \\ 
 \hline
    c\textsubscript{vbnd}& $[2] \times [N] \times [0,E]$ array &   \makecell{  c$^T$\textsubscript{vbnd}$[1,j]X = \Omega-v_j$, for all $1 \leq j \leq N$, and \\  c\textsubscript{vbnd}$[2,j]X = v_j$} \\ 
 \hline
   c\textsubscript{sum}& $[2]\times [0,E]$ array  &  c$^T$\textsubscript{sum}$[i]X = (-1)^{i+1} \sum_{j=1}^N (X_j + X_{N+j})$, for $i = 1,2$\\ 
 \hline
   d\textsubscript{sum}& [2]-vector  & d\textsubscript{sum}$[i] = (-1)^iN/2$, for $i=1,2$ \\ 
 \hline
    c\textsubscript{mean}& $[0,E]$-vector  &  c$^T$\textsubscript{mean}$X = L\sum_{j=1}^N(jX_j - (j-1)X_{j-1})$ \\ 
 \hline
   d\textsubscript{mean}& Real  & $-h_1N/2$ \\ 
 \hline
     c\textsubscript{mome}& $[0,E]$-vector  &  c$^T$\textsubscript{mome}$X = -L^2\sum_{j=1}^N (j-1)^2(X_j+X_{N+j})$  \\ 
 \hline
   d\textsubscript{mome}& Real  & $(2/3 + h_2^2/2)N/2$ \\ 
 \hline
      c\textsubscript{sin-lower}& $ [2R] \times [0,E]$ array  & c$^T$\textsubscript{sin-lower}$X =-\frac{8}{m \pi}\sin\left( m\pi /2 \right) b_m -$Fsin\textsubscript{LB}$[m]$   \\ 
 \hline
   c\textsubscript{sin-upper}& $ [2R] \times [0,E]$ array   & c$^T$\textsubscript{sin-upper}$X =\frac{8}{m \pi}\sin\left( m\pi /2 \right) b_m+$Fsin\textsubscript{LB}$[m]$ \\ 
 \hline
      c\textsubscript{ckbnd}& $[2] \times [T] \times [0,E]$ array  &  c$^T$\textsubscript{ckbnd}$[i,k]X = (-1)^{i+1} c_k$ for $i=1,2$ and $1 \leq k \leq T$ \\ 
 \hline
   d\textsubscript{ckbnd}& $[2] \times [T] $ array  & d\textsubscript{ckbnd}$[i,k] = 2/\pi$ for $i = 1,2$ and $1 \leq k \leq T$ \\ 
 \hline
       c\textsubscript{dkbnd}& $[2] \times [T] \times [0,E]$ array  &  c$^T$\textsubscript{dkbnd}$[i,k]X = (-1)^{i+1} d_k$  for $i=1,2$ and $1 \leq k \leq T$ \\ 
 \hline
   d\textsubscript{dkbnd}& $[2] \times [T] $ array  &  d\textsubscript{dkbnd}$[i,k] = 2/\pi$ for $i = 1,2$ and $1 \leq k \leq T$ \\ 
 \hline
       c\textsubscript{c1bnd}& $[2] \times [0,E]$ array  &  c$^T$\textsubscript{c1bnd}$[i]X = (-1)^{i+1}c_1$ for $i = 1,2$ \\ 
 \hline
   d\textsubscript{c1bnd}& [2]-vector& d\textsubscript{ckbnd}$ = [-p_1,p_2]$ \\ 
 \hline
       c\textsubscript{d1bnd}& $[2] \times [0,E]$ array  &  c$^T$\textsubscript{d1bnd}$[i]X = (-1)^{i+1}d_1$ for $i = 1,2$ \\ 
 \hline
   d\textsubscript{d1bnd}& [2]-vector  &  d\textsubscript{dkbnd}$ = [-q_1,q_2]$ \\ 
 \hline
        c\textsubscript{ep}& $[2] \times [R] \times [0,E]$ array  &  c$^T$\textsubscript{ep}$[i,m]X = (-1)^{i+1}\epsilon_m$ for $i = 1,2$ and $1 \leq m \leq R$\\ 
 \hline
   d\textsubscript{ep}& $[2] \times [R]$ array & d\textsubscript{ckbnd}$[i,m] = \frac{1}{4-m^2/T^2} \cdot \frac{2m}{\pi \sqrt{6T^3}}$ for  $i = 1,2$ and $1 \leq m \leq R$\\ 
 \hline
       c\textsubscript{del}& $[2] \times [R]\times [0,E]$ array  &  c$^T$\textsubscript{del}$[i,m]X = (-1)^{i+1}\delta_m$ for $i = 1,2$ and $1 \leq m \leq R$ \\ 
 \hline
   d\textsubscript{del}& $[2] \times [R]$ array  &  d\textsubscript{dkbnd}$[i,m] =  \frac{1}{4-m^2/T^2}\cdot \frac{4}{\pi\sqrt{2T}} $ for $i = 1,2$ and $1 \leq m \leq R$ \\ 
 \hline
        c\textsubscript{cosup}& $[0,E]$-vector  &  c$^T$\textsubscript{cosup}$X =  $Fcos\textsubscript{UB}$[2]$ \\ 
 \hline
   d\textsubscript{cosup}& Real & d\textsubscript{cosup}$=\frac{N}{2}(p_2^2 + \max\{q_1^2,q_2^2\})$. \\ 
 \hline

\end{tabular}

\caption{Data for SOCP}
\label{socpdata}
\end{table}

\normalsize

\clearpage

\section{Appendix II: Dual SOCP}

The dual of a SOCP can be concisely stated using the data from the primal, see Section 4.1 of \cite{LO} for details on dual SOCPs. The standard formulation of a dual SOCP involves equality constraints. We will eliminate some variables of the dual SOCP so that no equality constraints are used. This will allow us to check a particular variable assignment is far enough inside the interior of the dual space, thereby guaranteeing that assignment is feasible in spite of worst-case floating point arithmetic errors. We begin by reviewing the formulation of a SOCP and its dual. The standard formulation of a SOCP takes the form 
\begin{align}\label{primalform}
\textsc{minimize: } &\Phi^Tx  \nonumber \\
\textsc{subject to: } &\| A_ix + b_i \|  \leq c_i^Tx+d_i, \quad 1 \leq i \leq N ,
\end{align}
where $x$ is a $n$-vector variable, $\Phi \in \mathbb{R}^n$, $A_i$ is an $n_i \times n$ matrix, $b_i \in \mathbb{R}^{n_i}$, $c_i \in \mathbb{R}^n$, and $d_i \in \mathbb{R}$. We have formulated our Appendix I primal SOCP in this way. The dual to \eqref{primalform} is

\begin{align}\label{dualform}
\textsc{maximize: } &- \sum_{i=1}^N(b_i^Tz_i + d_iy_i)  \nonumber \\
\textsc{subject to: } & \sum_{i=1}^N (A_i^Tz +c_iy_i) = \Phi \nonumber \\
&\| z_i \| \leq y_i, \quad 1 \leq i \leq N ,
\end{align}
where $z_i$ is a $n_i$-vector of variables, and $y_i$ is a nonnegative variable. There are $n$ linear equations forming the equality constraints of \eqref{dualform}. For our SOCP dual, we need to eliminate the equality constraints. Suppose that for some row $1 \leq j \leq n$ of the system of equality constraints in \eqref{dualform} there is an index $1 \leq k \leq N$ such that vector $c_k$ is all zero except for a nonzero entry at row $j$, then we can solve for the variable $y_k$ in terms of the other variables: 
\begin{equation}\label{elimdual} y_k= \frac{1}{c_k[j]} \big(\Phi_j - \sum_{\substack{1 \leq i \leq N \\ i \neq k }} (A^T_iz + c_iy_i )[j]\big) \geq 0.\end{equation}
After adding the inequality on the righthand side of \eqref{elimdual} to the constraints of program \eqref{dualform}, we can eliminate variable $y_k$ and the row $j$ equality constraint from program \eqref{dualform}. Our primal in Appendix I has $E+1 = 2N + 2T + 2R + 1$ variables, and so the corresponding dual has $E+1$ equality constraints, corresponding to the entries of vector $\Phi$. For each $0 \leq j \leq E$, there is a `c' vector with a unique nonzero entry in the $j^{th}$ row. This collection of `c' vectors allows us to eliminate all equality constraints. The `c' vectors we use to eliminate the equality constraints are described in Table~\ref{celim}.

\clearpage

\subsection{The SOCP dual}\label{socpdual}

\noindent\textsc{Variables: } See Table \ref{socpdualvar}.

\vspace{0.15cm}

\noindent\textsc{Variable expressions: } See Table \ref{socpdualvarex}.

\vspace{0.15cm}

\noindent\textsc{Objective: maximize} obj (defined in Table \ref{socpdualvarex}). 

\vspace{0.15cm}

\noindent\textsc{Constraints: } \begin{align*}  \text{ye\textsubscript{obnd}} & \geq 0 \\
\text{ye\textsubscript{wbnd\_2}}[j] & \geq 0 \text{ for all } 1 \leq j \leq N \\
\text{ye\textsubscript{vbnd\_2}}[j] & \geq 0 \text{ for all } 1 \leq j \leq N \\
\text{ye\textsubscript{ckbnd\_2}}[k] & \geq 0 \text{ for all } 1 \leq k \leq T \\
\text{ye\textsubscript{dkbnd\_2}}[k] & \geq 0 \text{ for all } 1 \leq k \leq T \\
\text{ye\textsubscript{ep\_2}}[m] & \geq 0 \text{ for all } 1 \leq m \leq R \\
\text{ye\textsubscript{del\_2}}[m] & \geq 0 \text{ for all } 1 \leq m \leq R \\
 (\text{y\textsubscript{coscone}}[m])^2 & \geq \sum_{i=1}^3 ( \text{z\textsubscript{coscone}}[m,i] )^2 \text{ for all } 1 \leq m \leq 2R\\
  (\text{y\textsubscript{par}})^2 & \geq \sum_{k=1}^{2T} ( \text{z\textsubscript{par}}[k] )^2 . 
\end{align*}

\renewcommand{\arraystretch}{1.5}

\begin{table}[h!] \centering
\scriptsize
 \begin{tabular}{| c |c|}

 \hline
          Name & Index of nonzero row   \\ 
 \hline
 \hline
    
    c\textsubscript{obnd} & 0 \\
    \hline
    c\textsubscript{wbnd}$[2,j]$ & $j$ for $1 \leq j \leq N$ \\ 
    \hline
    c\textsubscript{vbnd}$[2,j]$ & $N+j$ for $1 \leq j \leq N$ \\ 
    \hline
    c\textsubscript{ckbnd}$[2,j]$ & $2N+j$ for $1 \leq j \leq T$ \\ 
    \hline
    c\textsubscript{dkbnd}$[2,j]$ & $2N+T+j$ for $1 \leq j \leq T$ \\ 
    \hline
    c\textsubscript{ep}$[2,j]$ & $2N+2T+j$ for $1 \leq j \leq R$ \\ 
    \hline
    c\textsubscript{del}$[2,j]$ & $2N+2T+R+j$ for $1 \leq j \leq R$ \\

\hline
\end{tabular}

\caption{c-vectors with one nonzero entry}
\label{celim}
\end{table}

\renewcommand{\arraystretch}{1.5}

\begin{table}[h!] \centering
\scriptsize
 \begin{tabular}{| c |c|}

 \hline
          Name & Description  \\ 
 \hline
 \hline
    
    z\textsubscript{coscone} & $[2R] \times [3]$ variable array \\
    y\textsubscript{coscone} & $[2R]$-vector of nonnegative variables \\
    \hline
        z\textsubscript{par} & $[2T]$-vector of variables\\
    y\textsubscript{par} & nonnegative variable \\
    \hline
        y\textsubscript{wbnd\_1} &$[N]$-vector of nonnegative variables\\
      \hline
              y\textsubscript{vbnd\_1} &$[N]$-vector of nonnegative variables\\
      \hline
        y\textsubscript{sum} &  $[2]$-vector of nonnegative variables\\
          \hline
        y\textsubscript{mean} & nonnegative variable \\
          \hline
        y\textsubscript{mome} & nonnegative variable \\
          \hline
        y\textsubscript{sin-lower} &  $[2R]$-vector of nonnegative variables \\
          \hline
        y\textsubscript{sin-upper} &  $[2R]$-vector of nonnegative variables\\
          \hline
        y\textsubscript{ckbnd\_1} & $ [T]$-vector of nonnegative variables \\
          \hline
        y\textsubscript{dkbnd\_1} & $[T]$-vector of nonnegative variables\\
          \hline
        y\textsubscript{c1bnd} &$[2]$-vector of nonnegative variables \\
          \hline
        y\textsubscript{d1bnd} & $[2]$-vector of nonnegative variables \\
          \hline
        y\textsubscript{ep\_1} & $ [R]$-vector of nonnegative variables \\
          \hline
        y\textsubscript{del\_1} & $[R]$-vector of nonnegative variables\\
          \hline
        y\textsubscript{cosup} & nonnegative variable \\
  
\hline
\end{tabular}

\caption{Variables for dual SOCP}
\label{socpdualvar}
\end{table}

\renewcommand{\arraystretch}{1.5}

\begin{table}[h] \centering
\scriptsize
 \begin{tabular}{| c |c|c|}

 \hline
          Name & Description & Definition  \\ 
 \hline
 \hline

        ATz\textsubscript{coscone} & $[0,E]$-expression vector & ATz\textsubscript{coscone}$[j] = \sum_{m=1}^{2R} \sum_{i=1}^3 $A\textsubscript{coscone}$[m,i,j]$z\textsubscript{coscone}$[m,i]$\\
      \hline
           cy\textsubscript{coscone} &$[0,E]$-expression vector &cy\textsubscript{coscone}$[j] = \sum_{m=1}^{2R}$c\textsubscript{coscone}$[m,j]$y\textsubscript{coscone}$[m]$ \\
      \hline
          ATz\textsubscript{par} &$[0,E]$-expression vector &ATz\textsubscript{par} $[j]= \sum_{k=1}^{2T}$A\textsubscript{par}$[k,j]$z\textsubscript{par}$[k]$ \\
      \hline
             cy\textsubscript{wbnd\_1} &$[0,E]$-expression vector & cy\textsubscript{wbnd\_1} $[j]=\sum_{i=1}^N$c\textsubscript{wbnd}$[1,i,j]$y\textsubscript{wbnd\_1}$[i]$ \\
      \hline
             cy\textsubscript{vbnd\_1} &$[0,E]$-expression vector & cy\textsubscript{wbnd\_1} $[j]=\sum_{i=1}^N$c\textsubscript{vbnd}$[1,i,j]$y\textsubscript{vbnd\_1}$[i]$ \\
      \hline
             cy\textsubscript{sum} &$[0,E]$-expression vector & $\sum_{i=1}^2$cy\textsubscript{sum} $[i,j]=$c\textsubscript{sum}$[i,j]$y\textsubscript{sum}$[i]$ \\
      \hline
             cy\textsubscript{mean} &$[0,E]$-expression vector & cy\textsubscript{mean} $[j]=$c\textsubscript{mean}$[j]$y\textsubscript{mean} \\
      \hline
             cy\textsubscript{mome} &$[0,E]$-expression vector & cy\textsubscript{mome} $[j]=$c\textsubscript{mome}$[j]$y\textsubscript{mome} \\
      \hline
               cy\textsubscript{sine} &$[0,E]$-expression vector &\makecell{ cy\textsubscript{sine}$[j] = \sum_{m=1}^{2R}$c\textsubscript{sin-lower}$[m,j]$y\textsubscript{sin-lower}$[m]$\\$+\sum_{m=1}^{2R}$c\textsubscript{sin-upper}$[m,j]$y\textsubscript{sin-upper}$[m]$}\\
      \hline
           cy\textsubscript{cd} &$[0,E]$-expression vector & \makecell{ cy\textsubscript{cd}$[j] = \sum_{m=1}^{T}$c\textsubscript{ckbnd}$[1,k,j]$y\textsubscript{ckbnd\_1}$[k]$\\$+\sum_{k=1}^{T}$c\textsubscript{dkbnd}$[1,k,j]$y\textsubscript{ckbnd\_1}$[k]$\\ $+\sum_{i=1}^{2}($c\textsubscript{c1bnd}$[i,j]$y\textsubscript{c1bnd}$[i]+$c\textsubscript{d1bnd}$[i,j]$y\textsubscript{d1bnd}$[i])$}\\
      \hline
       cy\textsubscript{epdel} &$[0,E]$-expression vector &\makecell{ cy\textsubscript{epdel}$[j] = \sum_{m=1}^{R}$c\textsubscript{ep}$[1,m,j]$y\textsubscript{ep\_1}$[m]$\\$+\sum_{m=1}^{R}$c\textsubscript{del}$[1,m,j]$y\textsubscript{del\_1}$[m]$}\\
      \hline
         cy\textsubscript{cosup} &$[0,E]$-expression vector & cy\textsubscript{cosup} $[j]=$c\textsubscript{cosup}$[j]$y\textsubscript{cosup} \\
         
      \hline
         subtotal &$[0,E]$-expression vector & subtotal is the sum of all previous $[0,E]$-expression vectors in this table  \\
            \hline
            \hline

       ye\textsubscript{obnd} & expression & $1-$subtotal$[0]$ \\
      \hline

              ye\textsubscript{wbnd\_2} &$[N]$-expression vector & ye\textsubscript{wbnd\_2}$[j] = -$ subtotal$[j]$ for $1 \leq j \leq N$ \\
      \hline
             ye\textsubscript{vbnd\_2} & $[N]$-expression vector & ye\textsubscript{vbnd\_2}$[j] = -$ subtotal$[N+j]$ for $1 \leq j \leq N$ \\
      \hline

       ye\textsubscript{ckbnd\_2} & $ [T]$-expression vector& ye\textsubscript{ckbnd\_2}$[j] = -$ subtotal$[2N+j]$ for $1 \leq j \leq T$ \\
          \hline
        ye\textsubscript{dkbnd\_2} & $[T]$-expression vector& ye\textsubscript{dkbnd\_2}$[j] = -$ subtotal$[2N+T+j]$ for $1 \leq j \leq T$\\
    \hline
      ye\textsubscript{ep\_2} & $[R]$-expression vector & ye\textsubscript{ep\_2}$[j] = -$ subtotal$[2N+2T+j]$ for $1 \leq j \leq R$\\
          \hline
        ye\textsubscript{del\_2} & $[R]$-expression vector& ye\textsubscript{del\_2}$[j] = -$ subtotal$[2N+2T+R+j]$ for $1 \leq j \leq R$\\
          \hline
  \hline
  
        bTz\textsubscript{coscone} & expression &  bTz\textsubscript{coscone}$ = \sum_{m=1}^{2R} \sum_{i=1}^3 $b\textsubscript{coscone}$[m,i]$z\textsubscript{coscone}$[m,i]$\\
      \hline
                 dy\textsubscript{coscone} & expression &  dy\textsubscript{coscone}$ = \sum_{k=1}^{2R} $d\textsubscript{coscone}$[k]$y\textsubscript{coscone}$[k]$\\
      \hline
                       dy\textsubscript{par} & expression &  dy\textsubscript{par}$ =  $d\textsubscript{par}y\textsubscript{par}\\
      \hline
                     dy\textsubscript{obnd} & expression &  dy\textsubscript{obnd}$ = $d\textsubscript{obnd}ye\textsubscript{obnd}\\
                     \hline
               dy\textsubscript{sum} & expression &  dy\textsubscript{sum}$ = \sum_{i=1}^{2} $d\textsubscript{sum}$[i]$y\textsubscript{sum}$[i]$\\
      \hline
                  dy\textsubscript{mean} & expression &  dy\textsubscript{mean}$ = $d\textsubscript{mean}y\textsubscript{mean}\\
      \hline
                  dy\textsubscript{mome} & expression &  dy\textsubscript{mome}$ = $d\textsubscript{mome}y\textsubscript{mome}\\
      \hline
                    dy\textsubscript{cd} & expression &  \makecell{  dy\textsubscript{cd}$ = \sum_{i=1}^{2} ($d\textsubscript{c1bnd}$[i]$y\textsubscript{c1bnd}$[i] +$d\textsubscript{d1bnd}$[i]$y\textsubscript{d1bnd}$[i]$) \\ $+\sum_{k=1}^{T} \big($d\textsubscript{ckbnd}$[1,k]$y\textsubscript{ckbnd\_1}$[k] +$d\textsubscript{dkbnd}$[1,k]$y\textsubscript{dkbnd\_1}$[k]$ \\ $+ $d\textsubscript{ckbnd}$[2,k]$ye\textsubscript{ckbnd\_2}$[k] +$d\textsubscript{dkbnd}$[2,k]$ye\textsubscript{dkbnd\_2}$[k]$\big) } \\
      \hline
             dy\textsubscript{epdel} & expression &\makecell{  dy\textsubscript{epdel}$ = \sum_{m=1}^{R} ($d\textsubscript{ep}$[1,m]$y\textsubscript{ep\_1}$[m] +$d\textsubscript{del}$[1,m]$y\textsubscript{del\_1}$[m]$) \\ $+\sum_{m=1}^{R} ($d\textsubscript{ep}$[2,m]$ye\textsubscript{ep\_2}$[m] +$d\textsubscript{del}$[2,m]$ye\textsubscript{del\_2}$[m]$) }\\
      \hline
            dy\textsubscript{cosup} & expression &  dy\textsubscript{cosup}$ = $d\textsubscript{cosup}y\textsubscript{cosup}\\
      \hline
             obj & expression &  $-1$ times the sum of the previous 10 expressions\\
      \hline

\end{tabular}

\caption{Variable expressions for dual SOCP}
\label{socpdualvarex}
\end{table}

\subsection{Feasibility verification} 

We will use a numerical solver to find feasible points in the dual SOCP, thereby giving lower bounds to the primal optimum. Given an assignment of values to the variables of the dual, we need to verify that in spite of any floating point arithmetic errors, the constraints are satisfied. We use IBM's CPLEX optimization software, which employs a barrier algorithm to solve SOCP. CPLEX uses double-precision (64-bit) arithmetic in its computations with a rounding error unit of $u = 2^{-53}$. For any $x \in \mathbb{R}$, let $\fl(x)$ be its stored double-precision value. If no overflow or underflow occurs, we have the following bounds. 

\begin{itemize}
\item Let $x,y \in \mathbb{R}$ and $\star \in \{ +,-,\times,\div\}$ be an operation where $y = 0$ and $\star = \div$ do not both hold, then for sum $|\delta| \leq u$:
\[\fl(x \star y) = (x \star y)(1+\delta) .\]
\item Let $n \in \mathbb{N}$. If $x_1,\ldots,x_n \in \mathbb{R}$, and $S_n = \sum_{k=1}^n x_k$, then
\[ \big| S_n - \fl (S_n) \big| \leq \frac{(n-1)u}{1-(n-1)u} \sum_{k=1}^n |x_k|.\]
\end{itemize}

Details of these estimates, and improvements on them can be found in  \cite{HIG}. Let $n \in \mathbb{N}$, and $x_k,y_k \in \mathbb{R}$ for $k=1,\ldots,n$. The arithmetic expressions involved in the constraints of our dual SOCP take the form 
\[ S= \sum_{k=1}^n x_ky_k .\]
By the earlier two floating point estimates we have: 
\begin{equation}\label{FLerr} \big| S - \fl (S) \big| \leq \frac{(n-1)u}{1-(n-1)u} \sum_{k=1}^n |\fl(x_iy_i)| \leq \frac{(n-1)u(1+u)}{1-(n-1)u} \sum_{k=1}^n |x_iy_i| .\end{equation}
After receiving a proposed assignment of dual variable values from CPLEX, we verify that every constraint inequality is strictly satisfied by a margin of at least the above error bound. The error bound on the righthand side of \eqref{FLerr} needs to be computed. We did this by roughly estimating the individual terms. For example, consider the final constraint in our dual SOCP: 
\begin{equation}\label{findualcon}   (\text{y\textsubscript{par}})^2 - \sum_{k=1}^{2T} ( \text{z\textsubscript{par}}[k] )^2  \geq 0. \end{equation}
For all feasible solutions to the dual SOCP we use in this paper, the maximum absolute value of $\{\text{z\textsubscript{par}}[k]\}_{k=1}^{2T}$ is less than 0.0005. Therefore, the error bound for \eqref{findualcon} is less that $10^{-13}$. On the other hand, the minimum value of the lefthand side of \eqref{findualcon} for all feasible solutions we use exceeds $10^{-11}$.

\subsection{Reusing a feasible solution} 

By inspecting Table~\ref{socpdata} we see that the parameters $h_1,h_2,p_1,p_2$ are only used in the `d' type data (never in A, b, or c). From \eqref{primalform} and \eqref{dualform} we see this means that $h_1,h_2,p_1,p_2$ will only affect the objective function of the dual, but not the constraints. Hence any single feasible solution to the dual, will give a lower bound on the optimum for any choice of $h_1,h_2,p_1,p_2$. We make this precise below.

\begin{lemma}\label{reuselem} Fix a choice of input $N,T,R,h_1,h_2,p_1,p_2,q_1,q_2$ for the dual program. Suppose there exists a feasible solution for this input and let $\ol{\emph{y}}$\textsubscript{mean},$\ol{\emph{y}}$\textsubscript{mome},$\ol{\emph{y}}$\textsubscript{c1bnd}, and $\ol{\emph{y}}$\textsubscript{cosup} be the values of the dual variables ${\emph{y}}$\textsubscript{mean},${\emph{y}}$\textsubscript{mome},${\emph{y}}$\textsubscript{c1bnd}, and ${\emph{y}}$\textsubscript{cosup} for this solution. Let $h_1',h_2',p_1',p_2' \in \mathbb{R}$ and $\emph{obj}(h_1',h_2',p_1',p_2')$ be the value of the objective function for this solution and input choice $N,T,R,h_1',h_2',p_1',p_2',q_1,q_2$. If $\emph{opt}(h_1',h_2',p_1',p_2')$ is the optimum of the dual program with input $N,T,R,h_1',h_2',p_1',p_2',q_1,q_2$ then 

\begin{align} \emph{opt}(h_1',h_2',p_1',p_2') &  \geq \emph{obj}(h_1,h_2,p_1,p_2)\nonumber \\ 
& +\frac{N}{4} \left( 2(h_1'-h_1)\ol{\emph{y}}\textsubscript{mean} + (h_2^2-(h_2')^2)\ol{\emph{y}}\textsubscript{mome} + 2(p_2^2-(p_2')^2\ol{\emph{y}}\textsubscript{cosup}\right)\nonumber \\
  & \quad + (p_1'-p_1)\ol{\emph{y}}\textsubscript{c1bnd}[1] + (p_2-p_2')\ol{\emph{y}}\textsubscript{c1bnd}[2] . \label{changeinobj}\end{align}
\end{lemma}
\begin{proof} From the definitions of Table~\ref{socpdata} and Table~\ref{socpdualvarex} the value of \text{obj}$(h_1',h_2',p_1',p_2') - \text{obj}(h_1,h_2,p_1,p_2)$ is precisely the quantity on the second two lines of \eqref{changeinobj}. Since opt$(h_1',h_2',p_1',p_2') \geq $ obj$(h_1',h_2',p_1',p_2')$ we have the claimed result.

\end{proof}

We now discuss how Lemma~\ref{reuselem} is applied. We can find a feasible point for our dual SOCP with input $N = 25000, T = 7000, R = 10$ and 
\[ h_1 = h_2 = 0.015, \quad p_1 = p_2 = 0.385, \quad  -q_1 =q_2 =  0.02, \quad \text{obj} = 0.37905.\]
This feasible point makes the following variable assignments: 
\begin{align*}
\text{y\textsubscript{mean}} &= 0.0000014902 \\
\text{y\textsubscript{mome}} &= 0.00010235 \\
\text{y\textsubscript{cosup}} &= 0.000038011 \\
\text{y\textsubscript{c1bnd}} &= (0.35962,0) .\\
\end{align*}

Let opt$(h_1',h_2',p_1',p_2')$ be as defined Lemma~\ref{reuselem}. We can use this feasible point to determine a set of $h,p \in \mathbb{R}$ such that $h,p \in \mathbb{R}$ has opt$(h,h,p,p) \geq 0.379005$. By Lemma~\ref{reuselem}, if $h,p \in \mathbb{R}$ is such that

\begin{align} \frac{25000}{4} \big( 2(h-0.015)& \cdot 0.0000014902+ (0.015^2-h^2)\cdot 0.00010235  + 2(0.385^2-p^2) \cdot 0.000038011 \big) \nonumber \\
& +  (0.385-p) \cdot 0.35962 \geq 0.379005 - 0.37905, \label{geineq}
\end{align}
then we can conclude opt$(h,h,p,p) \geq 0.379005$. The set of $(h,p)$ that satisfy the above inequality lie inside an ellipse, shown in Figure~\ref{oneellip}.

\begin{figure}[h!]
  \centering
    \includegraphics[width=0.4\textwidth]{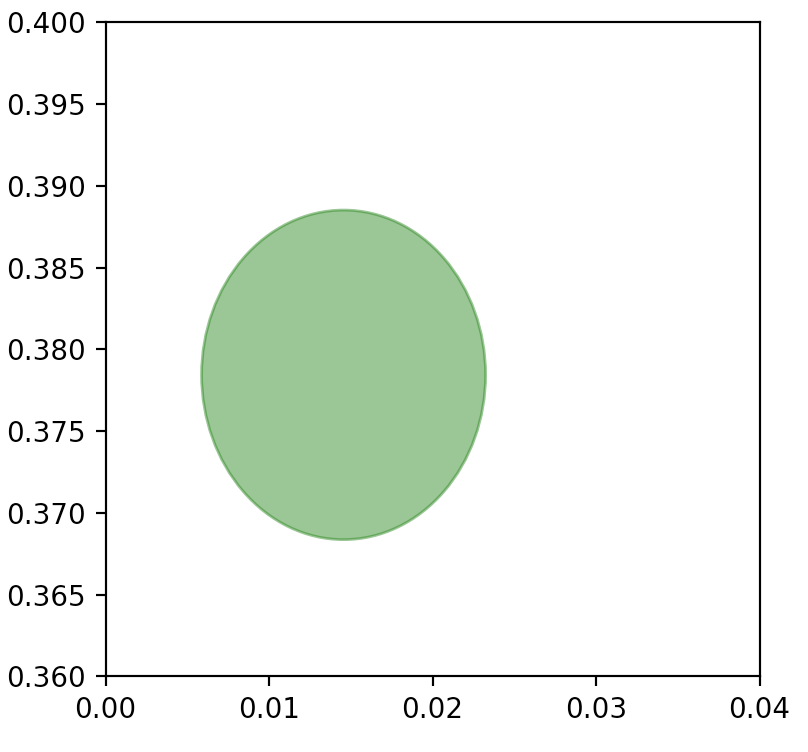}
    \caption{A subset of the region in $(h,p)$ space where opt$(h,h,p,p) \geq 0.379005$.  }
    \label{oneellip}
\end{figure}

This gives a consequence on the original $f,M$ functions of \eqref{fgm} by Proposition~\ref{convexprop}. For any $f,M$ as in  \eqref{fgm}, put 
\[ h = E(M),  \quad p = \int_{-1}^1 \cos(\pi x) f(x) dx, \quad \text{and} \quad q = \int_{-1}^1 \sin(\pi x) f(x) dx .\]
If $-0.02 \leq q \leq 0.02$ and $(h,p)$ defined above satisfy \eqref{geineq}, then $\|M\|_{\infty} \geq 0.379005$.

\end{document}